\documentclass[11pt]{amsart}
\usepackage{amsmath,amsthm,amssymb, amscd, amsfonts}
\usepackage[all]{xy}
\usepackage{leftidx}
\usepackage[latin1]{inputenc}
\usepackage{enumerate}
\usepackage[dvipsnames]{xcolor}

\theoremstyle{plain}

\newtheorem{theorem}{Theorem}[section]
\newtheorem{corollary}[theorem]{Corollary}

\newtheorem{proposition}[theorem]{Proposition}
\newtheorem{lemma}[theorem]{Lemma}

\theoremstyle{definition}

\newtheorem*{theorem1}{Theorem 4.7}
\newtheorem*{theorem2}{Theorem 5.6}

\theoremstyle{remark}
\newtheorem{remark}[theorem]{Remark}

\numberwithin{equation}{section}\theoremstyle{plain}

\newcommand{\I}{\mathcal{I}}

\newcommand{\B}{{\mathcal B}}
\newcommand{\C}{{\mathcal C}}
\newcommand{\D}{{\mathcal D}}

\newcommand{\Zz}{{\mathbb Z}}

\newcommand{\E}{{\mathcal E}}
\newcommand{\U}{{\mathcal U}}

\newcommand{\Rep}{\operatorname{Rep}}

\newcommand\Aut{\operatorname{Aut}}
\newcommand\Irr{\operatorname{Irr}}
\newcommand\FPdim{\operatorname{FPdim}}

\newcommand\vect{\operatorname{Vect}}
\newcommand\svect{\operatorname{sVect}}
\newcommand\id{\operatorname{id}}

\begin{document}
\title[Almost square-free modular categories]{On the classification of almost
square-free modular categories}
\author{Jingcheng Dong}
\address[Jingcheng Dong]{College of Engineering, Nanjing Agricultural
University, Nanjing 210031, China}
\email{dongjc@njau.edu.cn}

\author{Sonia Natale}
\address[Sonia Natale]{Facultad de Matem\'atica, Astronom\'\i a y F\'\i sica.
Universidad Nacional de C\'ordoba. CIEM -- CONICET. Ciudad
Universitaria. (5000) C\'ordoba, Argentina}
\email{natale@famaf.unc.edu.ar
\newline \indent \emph{URL:}\/ http://www.famaf.unc.edu.ar/$\sim$natale}

\keywords{Braided fusion category; modular category; group-theoretical fusion
category; braided $G$-crossed fusion category;
Tannakian category}

\subjclass[2010]{18D10; 16T05}

\date{November 1, 2017}

\begin{abstract}
Let $\C$ be a modular category of Frobenius-Perron dimension $dq^n$, where $q >
2$ is a  prime number and $d$ is a square-free integer. We show that $\C$ must
be integral and nilpotent and therefore group-theoretical. In the case where $q
= 2$, we describe the structure of $\C$ in terms of equivariantizations of
group-crossed braided fusion categories.
\end{abstract}

\maketitle

\section{Introduction}
\emph{Almost square-free} (ASF) fusion categories are a class of fusion
categories whose Frobenius-Perron dimensions factor simply. Recall that an ASF
fusion category is a fusion category of Frobenius-Perron dimension $dq^n$, where
$d$ is a square-free integer, $q$ is a prime number, $n$ is a non-negative
integer and ${\rm gcd}(q,d)=1$.

\medbreak
We shall work over an algebraically closed field $k$ of characteristic zero.
Recall that a \emph{modular} category is a braided fusion category with a ribbon
structure satisfying certain non-degeneracy condition (see Subsection
\ref{sec2.3} for a precise definition of this notion). Modular categories are
a relevant class of fusion categories due to their applications in areas like
low dimensional topology and conformal field theory. 

\medbreak 
The problem of classifying ASF modular categories has been addressed in several
papers. It was shown in \cite{drinfeld2007group} that every braided fusion
category of Frobenius-Perron dimension $q^n$ is group-theoretical. The
second-named author showed in \cite[Corollary 7.3]{witt-wgt} that every ASF
modular category is solvable. For distinct prime numbers $q > 2$, $p$ and
$r$, results on the structure of modular categories of Frobenius-Perron
dimensions $pq^n$, $pqr$, $4d$ and $8d$, where $1\leq n \leq 5$ and $d$ is an
odd natural number, were obtained in   \cite{ego, naidu-rowell,
bruillard2013classification, bpr, 2016DongIntegral,DoTu2015}.

In the present paper our first main result gives a nearly complete
classification of these modular categories. We prove the following theorem:

\begin{theorem1}
Let $q>2$ be a prime number and let $d$ be a square-free integer. Suppose that
$\C$ is a modular category such that $\FPdim \C = dq^n$, $n \geq 0$. Then $\C$
is integral and nilpotent.
\end{theorem1}

As an application of Theorem \ref{main} we show that an ASF modular category
of Frobenius-Perron dimension $dq^4$, where $q$ is a odd prime number not
dividing $d$ is necessarily pointed. See Corollary \ref{pq4}.

\medbreak 
By \cite{drinfeld2007group}, every integral nilpotent braided fusion category is
group-theoreti\-cal. Therefore Theorem \ref{main} implies that any ASF modular
category of the specified dimension, namely, except in the $q = 2$ integral
case,  is group-theoretical, entailing the classification of such modular
categories in group-theoretical terms, as explained below.

\medbreak
Group-theoretical fusion categories are important examples of fusion categories
whose objects have integer Frobenius-Perron dimensions. They are also
interesting examples of fusion categories that can be described explicitly in
terms of finite groups and their cohomology \cite{etingof2005fusion}.

\medbreak
Let $G$ be a finite group, $H \leq G$ a subgroup of $G$, $\omega:G\times G\times
G \to k^{\times}$ a normalized 3-cocycle, and $\psi: H \times H \to k^{\times}$ 
a normalized 2-cochain such that $\mathrm{d}\psi = \omega |_H$. Consider the
fusion category $\vect_G^\omega$ of finite dimensional $G$-graded vector spaces
with associativity given by 3-cocycle $\omega$. Then the twisted group algebra
$k_\psi[H]$ is  an associative algebra in $\vect_G^\omega$.
Thus the category
$\C(G, \omega, H, \psi)$ of $k_\psi[H]$-bimodules in $\vect_G^\omega$ is a
fusion category with tensor product $\otimes_{k_\psi[H]}$ and unit object
$k_\psi[H]$. A fusion category is called \emph{group-theoretical} if it is
equivalent to one of the form $\C(G, \omega, H, \psi)$ \cite[Section
8.8]{etingof2005fusion}.

\medbreak
Let $\C$ be a braided group-theoretical fusion category.
The braiding of $\C$ yields a canonical embedding of $\C$ into its Drinfeld
center $\mathcal{Z}(\C)$. Following the proof of \cite[Theorem
1.2]{natale2003group}, $\mathcal{Z}(\C)$ is equivalent to
$\mathcal{Z}(\vect_G^\omega)\cong\Rep(D^\omega G)$
for some finite group $G$ and $3$-cocycle $\omega$ on $G$, where $\Rep(D^\omega
G)$ is the representation category of the twisted quantum double $D^\omega G$
\cite{dpr}. Therefore, every braided group-theoretical fusion category may be
realized as a fusion subcategory of $\Rep(D^\omega G)$. One of the main results
of \cite{naidu2009fusion} describes all fusion subcategories of $\Rep(D^\omega
G)$. Therefore this result gives a complete description of all braided
group-theoretical fusion categories in group-theoretical terms. In the light of
these observations, (braided) group-theoretical fusion categories are well
understood.

An approach towards an explicit parameterization of the modular categories in
Theorem \ref{main} is discussed in Remark \ref{explicit}. 

\medbreak
If $\C$ is an integral ASF modular category of dimension $dq^n$ then $\C$ can be
obtained by a $G$-equivariantization from a nilpotent fusion category of
nilpotency class $2$ \cite{2016DongIntegral}, where $G$ is a $q$-group. Our
second task of this paper is to describe the structure of $\C$ when $\C$ is
\emph{strictly weakly integral}; that is, when there exists at least a simple
object $X$ such that $\FPdim X$ is not an integer.

Suppose that $\C$ is a modular category containing a Tannakian subcategory 
$\E \cong \Rep G$,
where $G$ is a finite group. Then $\C$ can be obtained as a
$G$-equivariantization of a braided $G$-crossed fusion category $\oplus_{g\in
G}\D_g$; see Subsection \ref{2.4} for a discussion of this fact.

Our second main result is the following theorem.

\begin{theorem2}
Let $\C$ be a strictly weakly integral ASF modular category. Then $\C$ fits into
one of the following classes:

(1)\, $\C$ is equivalent to a Deligne tensor product $\I\boxtimes\B$, where $\I$
is an Ising fusion category and $\B$ is a pointed modular  category.

(2)\, $\C$ is equivalent to a $G$-equivariantization of a braided $G$-crossed 
fusion category
$\oplus_{g\in G}\D_g$, where $G$ is a $2$-group and $\D_e$ is a pointed modular 
category.

(3)\, $\C$ is equivalent to a $G$-equivariantization of
a braided $G$-crossed  fusion category
$\oplus_{g\in G}\D_g$, where $G$ is a $2$-group and $\D_e\cong\I\boxtimes\B$,
where $\I$ is an Ising fusion category and $\B$ is a pointed modular  category.
\end{theorem2}

Recall that an \emph{Ising} braided category is a  non-pointed braided fusion
category of Frobenius-Perron dimension $4$. Every Ising braided category is
non-degenerate.  See \cite[Appendix B]{drinfeld2010braided} for a classification
of Ising braided categories.

\medbreak This paper is organized as follows. In Section \ref{sec2}, we recall
some basic results and prove some basic lemmas which will be used throughout. In
Section \ref{sec3}, we determine an upper bound for the order of the dimensional
grading group of a weakly integral modular category. This result will be used in
Sections \ref{sec4} and \ref{ASF_fu_ca}. In Section \ref{sec4}, we study the
nilpotency of an integral ASF modular category and give a proof of Theorem
\ref{main}. In Section \ref{ASF_fu_ca}, we study the structure of a strictly
weakly integral ASF modular category and give a proof of Theorem \ref{symm}.

\section{Preliminaries}\label{sec2}
The category of finite dimensional vector spaces
over $k$ will be denoted by $\vect$. All tensor categories will be assumed to be
strict, unless explicitly stated.
We refer the reader to \cite{etingof2005fusion}, \cite{etingof2011weakly},
\cite{drinfeld2010braided} for the notions on fusion categories and braided
fusion categories used throughout.

\subsection{Frobenius-Perron dimension}\label{sec2.1} Let $\C$ be a fusion
category. The
Frobenius-Perron dimension $\FPdim X$ of a
simple object $X \in \C$ is defined as the Frobenius-Perron eigenvalue of
the matrix of left multiplication by the class of $X$ in the basis $\Irr(\C)$ of
the Grothendieck ring of $\C$ consisting of isomorphism classes of simple
objects. The Frobenius-Perron dimension of $\C$ is $\FPdim \C =
\sum_{X \in \Irr(\C)} (\FPdim X)^2$. The fusion category $\C$ is called
\emph{integral}  if
$\FPdim X$ is a natural number, for all simple object $X \in \C$, and it is
called
\emph{weakly integral} if $\FPdim \C$ is a natural number. We shall say that
$\C$ is \emph{strictly weakly integral} if it is weakly integral but not
integral.

\subsection{Nilpotency of a fusion category}\label{sec2.2} Let $G$ be
a finite group and let $\C$ be a fusion category. A $G$-grading on a fusion
category $\C$ is a decomposition $\C =
\oplus_{g\in G} \C_g$, such that $\C_g \otimes \C_h \subseteq \C_{gh}$ and
$\C_g^* \subseteq \C_{g^{-1}}$, for all $g, h \in G$.
The fusion category $\C$ is called a \emph{$G$-extension} of a
fusion category $\D$ if there is a faithful grading $\C = \oplus_{g\in G} \C_g$
with neutral component $\C_e \cong \D$.

If $\C$ is any fusion category, there exists a finite group $\U(\C)$, called the
\emph{universal grading group} of $\C$, and a canonical faithful grading $\C =
\oplus_{g \in \U(\C)}\C_g$, with neutral component $\C_e = \C_{ad}$, where
$\C_{ad}$ is the \emph{adjoint subcategory} of $\C$, that is, the fusion
subcategory generated by $X\otimes X^*$, where $X$ runs over the simple objects
of $\C$.

\medbreak The \emph{descending central series} of $\C$ is the series of fusion
subcategories
\begin{equation}
\dots \subseteq \C^{(n+1)} \subseteq \C^{(n)} \subseteq \dots \subseteq \C^{(1)}
\subseteq \C^{(0)} = \C,
\end{equation}
where, for each $n \geq 0$, $\C^{(n+1)} = (\C^{(n)})_{ad}$.

\medbreak The fusion category $\C$ is \emph{nilpotent} if there exists $n \geq
0$ such that $\C^{(n)} \cong \vect$. Equivalently, $\C$ is nilpotent if there
exist a
sequence of fusion categories $\vect = \C_0 \subseteq \C_1 \dots \subseteq \C_n
= \C$, and a sequence of finite
groups $G_1, \dots, G_n$, such that for all $i = 1, \dots, n$, $\C_i$
is a $G_i$-extension of $\C_{i-1}$.

\subsection{Braided fusion categories}\label{sec2.3}
A braided fusion category is a fusion category endowed with a braiding, that is,
a
natural isomorphism $c_{X,Y} : X \otimes Y \rightarrow Y \otimes X$, $X, Y \in
\C$, subject to the hexagon axioms.

\medbreak Let $\C$ be a braided fusion category. If $\D$ is a fusion subcategory
of  $\C$,
the M\" uger centralizer $\D'$ of $\D$ in $\C$ is the full fusion
subcategory generated by objects $X \in \C$ such that $c_{Y, X}c_{X, Y} =
\id_{X \otimes Y}$, for all objects $Y \in \D$.

The M\" uger (or symmetric) center of $\C$ is the M\" uger centralizer
$\C'$. The category $\C$ is called \emph{symmetric} if $\C' = \C$.

\medbreak
Let $G$ be a finite group. The fusion category    $\Rep G$ of
finite dimensional representations of $G$ is a
symmetric fusion category with respect to the canonical braiding.
A braided fusion category $\E$ is called Tannakian, if $\E \cong \Rep G$ for
some finite group $G$, as symmetric fusion categories.

\medbreak  Every symmetric fusion category $\C$ is super-Tannakian, that is,
there exist a finite
group $G$ and a central element $u \in G$ of order $2$, such that $\C$ is
equivalent to
the category $\Rep(G, u)$ of representations of $G$ on finite-dimensional
super-vector spaces where $u$ acts as the parity operator.

Hence if $\C \cong \Rep(G, u)$ is a symmetric fusion category, then $\E = \Rep
G/u$ is the unique
Tannakian subcategory of $\C$ such that $\FPdim \E = \FPdim \C / 2$. Thus if
$\FPdim \C$ is bigger than $2$, then $\C$ necessarily contains a Tannakian
subcategory, and a non-Tannakian symmetric fusion category of Frobenius-Perron
dimension $2$ is equivalent to the category $\svect$ of finite-dimensional
super-vector spaces. See \cite[Subsection 2.12]{drinfeld2010braided}.

If $\C$ is any braided fusion category, its M\" uger center $\C'$ is a symmetric
fusion
subcategory of $\C$. $\C$ is  called
\emph{non-degenerate} (respectively, \emph{slightly degenerate})
if $\C' \cong \vect$ (respectively, if $\C' \cong \svect$). A  \emph{modular}
category is a non-
degenerate braided fusion category with a ribbon structure. By \cite[Proposition
8.23, Proposition 8.24]{etingof2005fusion}, a weakly integral braided fusion
category is modular if and only if it is non-degenerate.

\medbreak Suppose $\C$ is a braided nilpotent fusion category. It is shown in
\cite[Theorem 1.1]{drinfeld2007group} that $\C$ admits a unique decomposition
into a tensor product of braided fusion categories
\begin{equation}\label{dec-bnilp} \C_{p_1} \boxtimes \dots \boxtimes \C_{p_n},
\end{equation}
where for all $1\leq i \leq n$, $p_i$ is a prime number such that the
Frobenius-Perron dimension of $\C_{p_i}$ is a power of $p_i$ and the primes
$p_1, \dots, p_n$ are pairwise distinct.

\begin{remark}\label{rmk-mod-dec} Consider the decomposition \eqref{dec-bnilp}.
If $\C$ is non-degenerate, then the categories $\C_{p_1}, \dots, \C_{p_n}$ are
non-degenerate as well. In fact, for all $1\leq i \leq n$, we have $\FPdim \C =
\FPdim \C_{p_i} \FPdim \C_{p_i}'$. Therefore $\FPdim \C_{p_i}$ and $\FPdim
\C_{p_i}'$ are relatively prime. Hence the M\" uger center of $\C_{p_i}$, which
coincides with $\C_{p_i} \cap \C_{p_i}'$ must be trivial, that is, $\C_{p_i}$ is
non-degenerate.
\end{remark}

\medbreak
We next prove some results that will be needed in the course of the proof of
Theorem \ref{main}.

\begin{lemma}\label{gen-nil} Let $\C_1$, $\C_2$ be fusion subcategories of a
braided fusion category $\C$ and let $\C_1 \vee \C_2$ be the fusion subcategory
generated by $\C_1$ and $\C_2$. Then $(\C_1\vee \C_2)^{(n)} = \C_1^{(n)}\vee
\C_2^{(n)}$, for all $n \geq 0$. \end{lemma}

\begin{proof} Observe that if a fusion category $\D$ is generated by objects
$X_1, \dots, X_r$, $r \geq 1$, then $\D_{ad}$ is generated by $X_1\otimes X_1^*,
\dots, X_r\otimes X_r^*$. This implies that
	the adjoint subcategory of $\C_1\vee \C_2$ coincides with the fusion
subcategory $(\C_1)_{ad}\vee (\C_2)_{ad}$ generated by $(\C_1)_{ad}$ and
$(\C_2)_{ad}$. The lemma follows by induction on $n \geq 0$.
\end{proof}

Suppose that $\C$ is a fusion category such that the Grothendieck ring ${\rm
K}_0(\C)$ is commutative (e.g. if  $\C$ is braided). Let
$\D$ be a fusion subcategory of $\C$. Recall that the commutator $\D^{co}$ is a
fusion subcategory of $\C$, where $\D^{co}$ is the fusion subcategory of $\C$
generated by all objects $X$ such that $X \otimes X^*$ is an object of $\D$. It
follows from \cite[Lemma 4.15]{gelaki2008nilpotent} that 
$(\D^{co})_{ad}\subseteq \D\subseteq(\D_{ad})^{co}$.

\begin{proposition}\label{c-nil} Let $\C$ be a braided fusion category. Then the
following hold:
	
\smallbreak
(i) $\C$ contains a unique maximal nilpotent fusion subcategory $\C_{nil}$.
	
\smallbreak
(ii) Suppose $\C$ is non-degenerate. Let $\D \subseteq \C$ be the M\" uger
centralizer of the fusion subcategory $\C_{nil}$. Then $\D_{ad} = \D$, in other
words, $\U(\D) = 1$.	
\end{proposition}

\begin{proof} Suppose that $\C_1$ and $\C_2$ are fusion subcategories of $\C$
such that $\C_1$ and $\C_2$ are nilpotent. It follows from Lemma \ref{gen-nil}
that the fusion subcategory $\C_1 \vee \C_2$ generated by $\C_1$ and $\C_2$ is
nilpotent. Hence the fusion subcategory $\C_{nil}$ generated by all nilpotent
fusion subcategories of $\C$ is nilpotent. This proves (i).
	
\medbreak Now assume that $\C$ is non-degenerate. By \cite[Proposition
3.25]{drinfeld2010braided}	we have that $(\D_{ad})' = (\D')^{co}$. Since
$\C$ is non-degenerate, then $\D' = \C_{nil}'' = \C_{nil}$ \cite[Theorem
3.10]{drinfeld2010braided}, so that $(\D_{ad})' = (\C_{nil})^{co}$.

On the other hand, $((\C_{nil})^{co})_{ad} \subseteq \C_{nil}$, hence
$(\C_{nil})^{co}$ is nilpotent. Since $\C_{nil} \subseteq (\C_{nil})^{co}$, we
obtain that $\C_{nil} = (\C_{nil})^{co}$, by maximality of $\C_{nil}$. Then
$(\D_{ad})' = \C_{nil}$ and therefore $\D_{ad} = (\D_{ad})'' = (\C_{nil})' =
\D$.
This proves part (ii) and finishes the proof of the proposition.
\end{proof}

\subsection{Braided $G$-crossed fusion categories and Tannakian subcategories of
a braided fusion category}\label{2.4}

Let $G$ be a finite group.
Let us recall the correspondence
between equivalence classes of braided fusion categories containing $\Rep G$ as
a
Tannakian subcategory and equivalence classes of braided $G$-crossed  fusion
categories \cite{muger2004galois}, \cite[Section 4.4]{drinfeld2010braided}.

\medbreak
A \emph{braided $G$-crossed fusion
category} is a fusion category $\C$ endowed with a $G$-grading $\C
= \oplus_{g \in G}\C_g$ and an action of $G$ by tensor autoequivalences
$\rho:\underline G \to \Aut_{\otimes} \, \C$, such that $\rho^g(\C_h)
\subseteq
\C_{ghg^{-1}}$, for all $g, h \in G$, and a $G$-braiding $c: X \otimes Y \to
\rho^g(Y) \otimes X$, $g \in G$, $X \in \C_g$, $Y \in \C$, subject to
appropriate
compatibility conditions.

\medbreak Every $G$-crossed braided fusion category gives  rise, through the
equivariantization process, to a braided fusion category containing $\Rep G$ as
a Tannakian subcategory.
Conversely, suppose that $\E \cong \Rep G$ is a Tannakian subcategory of a
braided fusion category $\C$.
Then the de-equivariantization $\C_G$ of $\C$ with respect to $\E$ is a braided
$G$-crossed fusion category in a canonical way.

The neutral component $\C_G^0$ of $\C_G$ with respect to the associated
$G$-grading is a braided fusion category and the crossed action of $G$ on $\C_G$
induces an action of $G$ on $\C_G^0$ by braided auto-equivalences. Moreover,
there is an equivalence of braided fusion categories $(\C_G^0)^G \cong \E'$,
where $\E'$ is the centralizer in $\C$ of the Tannakian subcategory $\E$.

The braided fusion category $\C$ is non-degenerate if and only if the neutral
component $\C_G^0$ of $\C_G$ is non-degenerate and the $G$-grading of $\C_G$ is
faithful \cite[Proposition 4.6 (ii)]{drinfeld2010braided}.
In particular, if $\C$ is a non-degenerate braided fusion category
containing a Tannakian subcategory $\E \cong \Rep G$, then $|G|^2$ divides
$\FPdim \C$.

\medbreak
We shall use the following result:

\begin{theorem}{\cite[Theorem 4.1]{2016DongIntegral}.}\label{2016DongIntegral}
Let $q$ be a prime number and let $d$ be a square-free natural number. Let also
$\C$ be an integral modular category such that $\FPdim \C = dq^n$, $n \geq 0$.
Suppose that $\E \subseteq \C$ is a maximal Tannakian subcategory. Then $\E'$ is
group-theoretical. \qed
\end{theorem}

Note that if $\E \cong \Rep G$ is a maximal Tannakian subcategory, then $\E'
\cong (\C_G^0)^G$, where $\C_G^0$  is the core of $\C$ introduced in
\cite[Section 5]{drinfeld2010braided}. In fact, it is shown in
\cite{2016DongIntegral} that the core of $\C$ is a pointed non-degenerate
braided fusion category.

\section{The order of the dimensional  grading group of a weakly integral 
modular category}\label{sec3}
Let $\C$ be a weakly integral fusion category. It was shown in
\cite{gelaki2008nilpotent} that $\C=\oplus_{g\in E}\C_g$ is faithfully graded by
an elementary abelian $2$-group $E$. Moreover, there is a set of distinct
square-free integers $n_g$, $g\in E$,  such that $n_e=1$ and $\FPdim X\in
\mathbb{Z}\sqrt{n_g}$, for every $X\in \Irr(\C_g)$. This canonical grading is
called the \emph{dimensional grading} of $\C$. The neutral component $\C_e$ of
this grading is denoted by $\C_{int}$. The group $E$ will be called the
\emph{dimensional grading group} of $\C$.

The Frobenius-Perron dimension of every weakly integral fusion category has the
form $d2^n$, where $n$ is a non-negative integer and $d$ is an odd natural
number. The following theorem gives an upper bound for the order of $E$ in the
case where $\C$ is a modular category.

\begin{theorem}\label{uppbound}
Let $\C$ be a weakly integral modular category and let $\C=\oplus_{g\in E}\C_g$
be its dimensional grading. Suppose that  $\FPdim \C = d 2^n$, where $n \geq 0$
and $d$ is an odd natural number. Then $|E|\leq 2^{\frac{n}{2}}$. In other
words, $|E|^2$ divides the Frobenius-Perron dimension of $\C$.
\end{theorem}
\begin{proof}
The neutral component $\C_{int}$ is the unique component of this grading whose
objects are integral, that is, $\FPdim X\in\mathbb{Z}$ for every $X\in
\C_{int}$. Hence, the largest pointed fusion subcategory $\C_{pt}$ is a fusion
subcategory of $\C_{int}$ and therefore $\FPdim \C_{pt}$ divides $\FPdim
\C_{int}$, by \cite[Proposition 8.15]{etingof2005fusion}. On the other hand,
$\FPdim \C_{pt}=|\U(\C)|$ by \cite[Theorem 6.2]{gelaki2008nilpotent} and hence
$|\U(\C)|$ divides $\FPdim \C_{int}$. In addition, \cite[Corollary
3.7]{gelaki2008nilpotent} shows that there is a surjective group homomorphism
$\pi:\U(\C)\to E$. Hence $|E|$ divides $|\U(\C)|$.

Since $E$ is an elementary abelian $2$-group, we may assume that $|E|=2^t$ for
some non-negative integer $t$. From the arguments above, we have
$$\FPdim \C_{int}=|\U(\C)|m_2,\, 2^tm_1=|\U(\C)| \mbox{\, and \,}2^t\FPdim
\C_{int}=\FPdim \C,$$
where $m_1,m_2$ are positive integers. These equalities imply that
$2^{n-2t}d=m_1m_2$. Hence we get that $2^t=|E|\leq 2^{\frac{n}{2}}$.
\end{proof}

\begin{corollary}\label{2squarefree}
Let $\C$ be a weakly integral braided fusion category such that $\FPdim \C$ is
not divisible by $4$.
Then $\C$ is integral.
\end{corollary}

In other words, if $\C$ is a strictly weakly integral braided fusion category
then $4$ divides $\FPdim \C$.

\begin{proof}
We may assume that $\FPdim \C=d2^n$, where $d$ is odd and $n=0$ or $1$. If $n=0$
then $\FPdim \C$ is odd, and hence $\C$ is automatically integral by
\cite[Corollary 3.11]{gelaki2008nilpotent}. If $n=1$ and $\C$ is non-degenerate
then the order of $E$ is $1$ by Theorem \ref{uppbound}. This means that
$\C=\C_{int}$ is integral. In the rest of our proof, we shall consider the case
when $n=1$ and $\C$ is degenerate.

Suppose on the contrary that $\C$ is strictly weakly integral. In this case
\cite[Theorem 3.10]{gelaki2008nilpotent} shows that $\C$ has a faithful
$G$-grading, where $G$ is an elementary abelian $2$-group. Since $2\mid\FPdim
\C$ and $4\nmid\FPdim \C$, we have $G \cong \mathbb{Z}_2$ and hence
$\C=\C_0\oplus\C_1$. In addition, \cite[Theorem 3.10]{gelaki2008nilpotent} also
shows that $X\in\Irr(\C)$ is integral if and only if $X\in \C_0$. Notice that
$\FPdim \C_{0}=\FPdim \C/2$ is an odd integer.

Let $\E$ be the M\"{u}ger center of $\C$. It is not trivial by assumption. Since
$\E$ is symmetric, then it is integral and therefore $\E \subseteq \C_0$. Thus
$\FPdim \E$ must be odd, whence $\E\cong\Rep(H)$ is a Tannakian subcategory of
$\C$ for some finite group $H$. We may thus consider the de-equivariantization
$\C_H$ of $\C$ by $\Rep(H)$, which is a non-degenerate braided fusion category.

Observe that $\C_H$ is weakly integral and $\FPdim \C = |H| \, \FPdim \C_H$, so
that $\FPdim \C_H$ is not divisible by $4$. It follows from  the first part of
the proof that $\C_H$ is integral. Then $\C$ is integral because the class of
integral fusion categories is closed under taking equivariantization
\cite{burciu2013fusion}.
This completes the proof of the corollary.
\end{proof}

\begin{remark}
If 4 divides $\FPdim \C$,
then $\C$ may be not integral. An example of this class of modular categories is
classified in \cite{Bruillard20162364}.

On the other hand, Corollary \ref{2squarefree} combined with \cite[Lemma
1.2]{etingof1998some} implies that a weakly integral modular category such that
$\FPdim \C$ is square-free is necessarily pointed; see \cite[Lemma
3.4]{Bruillard20162364}.
\end{remark}

\section{Nilpotency of a class of ASF modular categories}\label{sec4}

Let $q$ be a prime number and let $d$ be a square-free natural number not
divisible by $q$. Throughout this section $\C$ will denote an ASF modular
category of Frobenius-Perron dimension $dq^n$, unless otherwise stated.

Recall from Proposition \ref{c-nil} that $\C$ contains a unique maximal
nilpotent fusion subcategory $\C_{nil}$. In what follows we shall denote by $\D
= \C_{nil}'$ the centralizer of $\C_{nil}$. By Proposition \ref{c-nil}, $\D_{ad}
= \D$.

Since $\C$ is modular, then \cite[Theorem 3.14]{drinfeld2010braided} implies
that
\begin{equation}
\FPdim \C_{nil} = aq^m, \quad \FPdim \D = bq^{n-m},
\end{equation} where $a$ and $b$ are natural numbers such that $ab = d$, and
$1\leq m \leq n$.

\begin{remark}\label{rmk-qpwr}
Observe that if $\E$ is a Tannakian subcategory of $\C$ then $(\FPdim \E)^2$
divides $\FPdim \C$. Hence $\FPdim \E$ is a power of $q$. In particular, every
Tannakian subcategory of $\C$ is nilpotent and therefore it is contained in
$\C_{nil}$.
\end{remark}

\emph{In the next lemmas we assume that $\C$ is integral.}

\begin{lemma}
The categories $\C_{nil}$ and $\D$ are group-theoretical.
\end{lemma}

\begin{proof} Since $\C_{nil}$ is nilpotent and integral, then it is
group-theoretical \cite{drinfeld2007group}.
Let $\E \subseteq \C$ be a maximal Tannakian subcategory. It follows from
Theorem \ref{2016DongIntegral} that  $\E'$ is group-theoretical.

As observed in Remark \ref{rmk-qpwr}, $\E \subseteq \C_{nil}$. Then $\D =
(\C_{nil})' \subseteq \E'$ and therefore $\D$ is also group-theoretical, as
claimed.
\end{proof}

Since $\D$ is a group-theoretical braided fusion category, it follows from
\cite[Theorem 7.2]{naidu2009fusion} that $\D$ contains a Tannakian subcategory
$\E \cong \Rep G$ such that the de-equivariantization $\D_G$ is a pointed fusion
category.

\begin{lemma}\label{dim-e}
Let $\E \cong \Rep G$ be a Tannakian subcategory of $\D$ such that $\D_G$ is 
pointed. Then $\E$ is contained in the M\" uger center of $\D$. In particular, 
$\D_G$ is a braided fusion category and the canonical dominant functor $F: \D
\to \D_G$ is a braided tensor functor.
\end{lemma}

\begin{proof}
Since $\E$ is a  Tannakian subcategory of $\C$, then we know that  $\FPdim \E =
q^j$, for some $1\leq j \leq n-m$, and $\E \subseteq \C_{nil} = \D'$ (see Remark
\ref{rmk-qpwr} (i)). In particular, $\E$ is contained in the M\"uger center of
$\D$, which coincides with $\D \cap \D'$. This implies that $\D_G$ is a braided
fusion category and the canonical functor $F: \D \to \D_G$ is a braided tensor
functor, as claimed.
\end{proof}

Suppose that $\FPdim \E = q^j$, $1\leq j \leq n-m$. Then $\FPdim \D_G =
bq^{n-m-j}$. By Lemma \ref{dim-e}, there is an action of the group $G$
on $\D_G$ by braided autoequivalences such that $\D \cong (\D_G)^G$.

Since $d$ is  by assumption square-free and not divisible by $q$, then we may
write $b = p_1\dots p_r$, where $p_1, \dots, p_r$ are pairwise distinct prime
numbers and $p_i \neq q$, for all $1\leq i \leq r$.

The fact that $\D_G$ is pointed implies that $\D_G$ contains unique fusion
subcategories $(\D_G)_{p_1}, \dots, (\D_G)_{p_r}$ and $(\D_G)_{q}$  such that
$\FPdim (\D_G)_{p_i} = p_i$, $1\leq i \leq r$, and $\FPdim (\D_G)_q =
q^{n-m-j}$. Moreover, there is an equivalence of braided fusion categories
\begin{equation}\label{dec-dg}\D_G \cong (\D_G)_{p_1} \boxtimes \dots \boxtimes
(\D_G)_{p_r} \boxtimes (\D_G)_{q}.\end{equation}
The fusion subcategories $(\D_G)_{p_i}$ and $(\D_G)_{q}$ are clearly stable
under the action of the group $G$, so that the equivariantizations $\D_{p_i}: =
(\D_G)_{p_i}^G$ and $\D_q: = (\D_G)_{q}^G$ are fusion subcategories of $\D$ such
that
$\FPdim \D_{p_i}^G = p_iq^j$ and $\FPdim \D_q^G = q^{n-m}$.

A dimension argument shows that $\D = \D_{p_1} \vee \dots \vee \D_{p_r} \vee
\D_q$. Since $\D_q$ is nilpotent, there exists $n \geq 1$ such that $\D_q^{(n)}
\cong \vect$. In view of Lemma \ref{gen-nil}, this implies  that $\D = \D^{(n)}$
is generated by $\D_{p_i}^{(n)}$, $1\leq i \leq r$. Then necessarily
\begin{equation}\D = \D_{p_1} \vee \dots \vee \D_{p_r}. \end{equation}

\begin{lemma}\label{dim} $\FPdim \D_{p_{i_1}} \vee \dots \vee \D_{p_{i_s}} =
q^jp_{i_1}\dots p_{i_s}$, for all pairwise distinct $1 \leq i_1, \dots, i_s \leq
r$, $s \geq 1$.	
\end{lemma}

\begin{proof}  The proof is by induction on $s$. The claim is clear if $s = 1$.
Assume $s > 1$. Notice that $\E \subseteq \cap_{i = 1}^r\D_{p_i}$, so that $q^j$
divides $\FPdim (\D_{p_{i_1}}\vee \dots \vee \D_{p_{i_{s-1}}}) \cap
\D_{p_{i_{s}}}$. By the inductive assumption,  $\FPdim (\D_{p_{i_1}}\vee \dots
\vee \D_{p_{i_{s-1}}}) \cap \D_{p_{i_{s}}}$ divides $q^j$ and therefore $\FPdim
(\D_{p_{i_1}}\vee \dots \vee \D_{p_{i_{s-1}}}) \cap \D_{p_{i_{s}}} = q^j$. The
identity follows by induction from \cite[Corollary 3.12]{drinfeld2010braided}.
\end{proof}

\begin{lemma}\label{dim-e-2}
(i) $\FPdim \E = q^{n-m}$.

\medbreak
(ii) The Tannakian subcategory $\E$ coincides with the M\" uger center of $\D$.
In particular, $\D_G$ is non-degenerate.
\end{lemma}

\begin{proof} It follows from Lemma \ref{dim} that
$$\FPdim \left(\D_{p_1} \vee \dots \vee \D_{p_r} \right) =
p_1\dots p_r q^j = bq^{j}.$$
Therefore $\FPdim \D = bq^{n-m} = \FPdim \left(\D_{p_1} \vee \dots \vee \D_{p_r}
\right) = bq^{j}$, which implies that $j = n-m$. This proves (i).

\medbreak
We now show (ii). Let $\B = \D \cap \D'$ be the M\" uger center of $\D$. By
Lemma \ref{dim-e}, $\E \subseteq \B$. On the other hand, $\FPdim \B$ divides
$\FPdim \D = bq^{n-m}$ and $\FPdim \D' = aq^m$, whence $\FPdim \B$ divides
$q^{n-m}$. In view of part (i), this implies that $\E = \B$, as claimed.
\end{proof}

We shall denote by $\rho: \underline G \to {\Aut}_{br}\D_G$ the action by
braided autoequivalences of $\D_G$ such that $\D \cong (\D_G)^G$.

Let also $(\Irr\D_G)_{p_i}$ be the set of isomorphism classes of simple objects
of $(\D_G)_{p_i}$, $1\leq i \leq r$. Thus $(\Irr\D_G)_{p_i}$ is a cyclic group
of order $p_i$ and the action $\rho$ induces an action of $G$ by group
automorphisms on $(\Irr\D_G)_{p_i}$, for all $1\leq i \leq r$.

\begin{lemma}\label{action-g}
The action of $G$ on  $(\Irr \D_G)_{p_i}$ is not trivial, for all $1\leq i \leq
r$.
\end{lemma}

\begin{proof}
Suppose on the contrary that there exists some $1\leq i \leq r$ such that the
action of $G$ on $(\Irr\D_G)_{p_i}$ is trivial, that is, $\rho^g(Y) \cong Y$,
for all simple objects $Y$ of $(\D_G)_{p_i}$. Since $\D_G$ is pointed, then
every such simple object $Y$ is invertible, and thus $Y^* \cong Y^{-1}$.

Let $X$ be a simple object of $\D_{p_i}$ and let $Y$ be a simple constituent of
$F(X)$ in $(\D_G)_{p_i}$.
Since the action of $G$ on $(\Irr\D_G)_{p_i}$ is trivial, then $F(X) \cong
Y^{(d)}$ \cite[Proposition 2.1]{burciu2013fusion}.
Therefore $$F(X \otimes X^*) \cong F(X) \otimes F(X)^* \cong Y^{(d)} \otimes
(Y^{-1})^{(d)} \cong \textbf{1}^{(d^2)}.$$
This implies that the essential image of $(\D_{p_i})_{ad}$ under the functor $F$
is the trivial subcategory $\langle \textbf{1} \rangle$ of $(\D_G)_{p_i}$.
Hence $(\D_{p_i})_{ad} \subseteq \E$.

\medbreak
Since $\D = \D_{ad} = (\D_{p_1})_{ad} \vee \dots \vee (\D_{p_r})_{ad}$, and  $\E
\subseteq \D_{p_l}$, for all $l = 1, \dots, r$, we obtain that
\begin{align*}\D & =  (\D_{p_1})_{ad} \vee \dots \vee (\D_{p_{i-1}})_{ad} \vee
(\D_{p_{i+1}})_{ad} \dots \vee (\D_{p_r})_{ad} \\ & \subseteq \D_{p_1} \vee
\dots \vee \D_{p_{i-1}} \vee \D_{p_{i+1}} \dots \vee \D_{p_r}, \end{align*}
This contradicts Lemma \ref{dim}, since $\FPdim \D = bq^n$ while  $$\FPdim
\D_{p_1} \vee \dots \vee \D_{p_{i-1}} \vee \D_{p_{i+1}} \dots \vee \D_{p_r} =
p_1 \dots p_{i-1}p_{i+1} \dots p_r q^n.$$
This contradiction shows that the action of $G$ on $(\Irr\D_G)_{p_i}$ must be
nontrivial, for all $1\leq i\leq r$, as claimed.
\end{proof}

We next combine the previous lemmas to prove the main result of this section.

\begin{theorem}\label{main}
Let $q$ be an odd prime number and let $d \geq 1$ be a square-free integer.
Suppose that $\C$ is an ASF modular category such that $\FPdim \C = dq^n$, $n
\geq 0$. Then $\C$ is integral and nilpotent.
\end{theorem}

\begin{proof}
Since $q$ is odd and $d$ is square-free, $\FPdim \C$ is not divisible by $4$,
hence $\C$ is integral by Corollary \ref{2squarefree}.

Suppose that $\C$ is not nilpotent. We may assume that $d$ is not divisible by
$q$ and $n \geq 1$.
	
By Lemmas \ref{dim-e}, \ref{dim-e-2} and \ref{action-g}, the M\" uger center
$\E$ of the category $\D = (\C_{nil})'$ is a Tannakian subcategory,  $\E \cong
\Rep G$, where $G$ is a group of order $q^{n-m}$. Furthermore, the
de-equivariantization $\D_G$ is a non-degenerate braided fusion category of
dimension $d$ and the action by braided autoequivalences of the group $G$ on
$\D_G$ induces a nontrivial action on $(\Irr\D_G)_{p_i}$, for all $1\leq i \leq
r$.
Then  $q$ must divide $p_i-1$, for all $1\leq i\leq r$. In particular,  the
primes $p_i$ and therefore also $b$ must be odd.
Notice that the categories $(\D_G)_{p_i}$, $1\leq i \leq r$, in the
decomposition \eqref{dec-dg} of $\D_G$ are also non-degenerate (see Remark
\ref{rmk-mod-dec}).

\medbreak
Let $p = p_i$, $1\leq i \leq r$.
Non-degenerate braided fusion categories of dimension $p$ are classified by
metric groups $(\mathbb F_p, \phi)$, where $\mathbb F_p$ denotes the field with
$p$ elements and $\phi$ is a non-degenerate quadratic from on $\mathbb F_p$.
Moreover, the quadratic form $\phi$ is one of the following:
\begin{equation}\label{phi}
\phi_1(a) = \zeta^{a^2}, \quad \text{or \; } \phi_2(a) = \zeta^{ca^2},
\end{equation}
where $\zeta \in k$ is some primitive $p$th. root of unity and $c \in \mathbb
F_p^\times$ is a quadratic nonresidue. See \cite[Proposition
A.6]{drinfeld2010braided}.

The action by braided autoequivalences $\rho: \underline G \to
{\Aut}_{br}(\D_G)_p$  corresponds to an action $\rho: G \to \Aut(\mathbb F_p)
\cong \mathbb F_p^\times$ of $G$ on $\mathbb F_p$ by group automorphisms
preserving the form $\phi$. In view of the possibilities \eqref{phi}, the
nontrivial action $\rho$ must satisfy
$$\rho(g) = \pm 1 (\text{mod }p),$$ for all $g \in G$. This implies that $G$ is
a $2$-group, that is, $q = 2$.
Thus we obtain that every non-degenerate braided fusion category of
Frobenius-Perron dimension $dq^n$ such that $q$ is odd is nilpotent, as claimed.
\end{proof}

One of the main results in \cite{drinfeld2007group} shows that an integral
nilpotent modular category is group-theoretical. As a consequence of  Theorem
\ref{main} we obtain:

\begin{corollary} Let $q$ be an odd prime number and let $d \geq 1$ be a
square-free integer. Then every ASF modular category of dimension $dq^n$, $n
\geq 0$, is group-theoretical. \qed
\end{corollary}

\begin{remark}\label{explicit} Let $\C$ be an ASF modular category. Suppose as
before that $\FPdim \C = dq^n$, where $q > 2$, $n \geq 1$ and $d$ is a
square-free natural number not divisible by $q$. Let $d = p_1 \dots p_r$ be the
decomposition of $d$ as a product of distinct prime numbers. Since $\C$ is
nilpotent, then 
\begin{equation*}\C \cong \C_{p_1} \boxtimes \dots \boxtimes \C_{p_r} \boxtimes
\C_{q^n},\end{equation*} where $\C_{p_i}$ is a (necessarily pointed) modular
category of Frobenius-Perron dimension $p_i$, $1\leq i \leq r$, and $\C_{q^n}$
is a modular category of Frobenius-Perron dimension $q^n$ \cite[Theorem
1.1]{drinfeld2007group}. 

Pointed modular categories of prime dimension $p$ are classified by metric
groups $(\mathbb F_p, \phi)$, where $\mathbb F_p$ denotes the field with $p$
elements and $\phi$ is a non-degenerate quadratic from on $\mathbb F_p$ as in
\eqref{phi}.

Therefore the problem of giving an explicit parameterization of all possible ASF
modular categories $\C$ under the mentioned restrictions reduces to the explicit
determination of modular categories of odd prime power Frobenius-Perron
dimension. This problem is beyond the scope of this paper, and will be postponed
for future consideration. 

However, in view of \cite[Theorem 5.11 and Proposition 6.7
(i)]{naidu2009fusion}, we can say that a modular category of Frobenius-Perron
dimension $q^n$ (which is necessarily group-theoretical
\cite{drinfeld2007group}) can be described by means of  a factorization $G = KH$
of a finite $q$-group $G$ into mutually centralizing normal subgroups $H$ and
$K$, together with a 3-cocycle $\omega$ on $G$ and a $G$-invariant
$\omega$-bicharacter $b: K \times H \to k^{\times}$ such that a certain
symmetric bicharacter associated to $b$ is non-degenerate on $H \cap K$.
\end{remark}

Our next result, Corollary \ref{pq4}, is an application of Theorem
\ref{main} in the dimension  $dq^4$  case. We shall use the following
lemmas\footnote{Lemma \ref{pt-ddqq} was motivated by a question of the
referee, that we acknowledge with thanks; it improves a version of Corollary
\ref{pq4} obtained in a previous version of this paper.}:

\begin{lemma}\label{pt-ddqq} Let $q$ be a prime number. Then the Drinfeld center
of a  fusion category of dimension $q^2$ is a pointed modular category. 
\end{lemma}

\begin{proof} 
Suppose that $G$ is any finite group and $\omega \in Z^3(G, k^\times)$ is a
3-cocycle on $G$. Then the Drinfeld center of the pointed fusion category
$\vect_G^\omega$ is equivalent as a braided fusion category to the category
$\Rep (D^\omega G)$ of finite dimensional representations of the twisted quantum
double $D^\omega G$ \cite{majid}.
	
Assume $G$ is abelian. Then the fusion category $\Rep (D^\omega G)$ is pointed
if and only if the class of the 3-cocycle $\omega$ belongs to the subgroup
$H^3(G, k^\times)_{ab}$ of  $H^3(G, k^\times)$, defined as $H^3(G,
k^\times)_{ab} = \cap_{x \in G} D_x$, where for all $x \in G$, $D_x: H^3(G,
k^\times) \to H^2(G, k^\times)$ is the group homomorphism that maps the class of
a 3-cocycle $\omega \in Z^3(G, k^\times)$ to the class of the 2-cocycle
$\omega_x \in Z^2(G, k^\times)$ defined in the form
\begin{equation*}\omega_x(g, h) = 
\frac{\omega(x, g, h)\; \omega(g, h, x)}{\omega(g, x, h)},
\end{equation*}
for all $g, h \in G$. See \cite[Corollary 3.6]{mason-ng}.
	
\medbreak Let now $\C$ be a fusion category of dimension $q^2$. Then
$\C$ is pointed and thus equivalent to the category $\vect_G^\omega$ for some
group $G$ of order $q^2$ and $\omega \in Z^3(G, k^\times)$. The group $G$ must
be isomorphic to one of the groups $G_1 = \langle c: \; c^{q^2} = 1 \rangle
\cong \Zz_{q^2}$ or $G_2 = \langle a, b: \; a^q = b^q = aba^{-1}b^{-1} =
1\rangle \cong \Zz_q \times \Zz_q$. 
	
\medbreak 
The group $H^3(G_1, k^\times) \cong \Zz_{q^2}$ is generated by the class of the
3-cocycle $\omega_{I}$ defined in the form 
\begin{equation*}\omega_{I}(c^{i}, c^{i'}, c^{i''}) = \xi^{i
\left[\frac{i'+i''}{q^2}\right]},
\end{equation*}
for all $0 \leq i, i', i'' \leq q^2-1$, where $\xi \in k^\times$ is a primitive
$q^2$th root of unity and for each $m \in \mathbb Q$, the notation $[m]$ indicates the largest integer less than $m$.
	
Observe that the generator $\omega_I$ satisfies the symmetry condition 
\begin{equation}\label{symm-w}
\omega_I(x, g, h) = \omega_I(x, h, g), \end{equation} for all $x, g, h \in G_1$.
This implies that for all $x \in G_1$ the class of the 2-cocycle $\omega_x$ is
trivial in $H^2(G_1, k^\times)$. We thus obtain that $H^3(G_1, k^\times) =
H^3(G_1, k^\times)_{ab}$ and therefore the category $\Rep (D^\omega G_1)$ is
pointed for every 3-cocycle $\omega$ on $G_1$.
	
\medbreak 
Regarding the group $G_2$, we have that $H^3(G_2, k^\times) \cong \Zz_q \times
\Zz_q \times \Zz_q$ is generated by the classes of the 3-cocycles 
$\omega_I^{(1)}$, $\omega_I^{(2)}$ and $\omega_{II}$ given, respectively, by the
formulas
\begin{align*}
& \omega_{I}^{(1)}(a^{i}b^{j}, a^{i'}b^{j'}, a^{i''}b^{j''}) = \zeta^{i
\left[\frac{i'+i''}{q}\right]}, \\
& \omega_{I}^{(2)}(a^{i}b^{j}, a^{i'}b^{j'}, a^{i''}b^{j''}) = \zeta^{j
\left[\frac{j'+j''}{q}\right]}, \\
& \omega_{II}(a^{i}b^{j}, a^{i'}b^{j'}, a^{i''}b^{j''}) = \zeta^{i
\left[\frac{j'+j''}{q}\right]},
\end{align*}	
for all $0 \leq i, i', i'', j, j', j'' \leq q-1$, where $\zeta \in k^\times$ is
a primitive $q$th root of unity. See for instance \cite[Section
2.3.2]{propitius}.
	
As in the case of $G_1$, we find that if $\omega$ is any of the generators
$\omega_I^{(1)}$, $\omega_I^{(2)}$ or $\omega_{II}$, then $\omega$ satisfies the
symmetry condition \eqref{symm-w} for all elements $x, g, h \in G_2$. Then also
in this case $H^3(G_2, k^\times) = H^3(G_2, k^\times)_{ab}$ and therefore the
category $\Rep (D^\omega G_2)$ is pointed for every 3-cocycle $\omega$ on $G_2$.
This finishes the proof of the lemma.
\end{proof}

\begin{lemma}\label{dim_q2}
Let $q$ be a prime number and let $\C$ be an integral  modular category of
dimension $q^4$.  Then $\C$ is pointed.
\end{lemma}

\begin{proof}
Suppose on the contrary that $\C$ is not pointed.
By \cite[Lemma 3.4]{2016DongIntegral}, $(\C_{ad})_{pt}$ is a symmetric
subcategory of dimension $q^2$. Since $\C$ is modular, the dimension of every
non-invertible simple object is $q$. Hence, $(\C_{ad})_{pt}=\C_{pt}$ is the
unique fusion subcategory of dimension $q^2$.
	
Case $q=2$. $\C$ is of type $(1,4;2,3)$ and hence it is a modular category of
rank $7$. By \cite[Theorem 5.8]{Bruillard20162364}, it should be pointed, a
contradiction.
	
Case $q>2$. In this case the symmetric category $\C_{pt}$ is necessarily
Tannakian.
By \cite[Theorem 3.2]{muger2003structure}, $\FPdim \C_{pt}'=q^2$. Since
$\C_{pt}$ is the unique fusion subcategory of dimension $q^2$, we get
$\C_{pt}'=\C_{pt}$ and hence the subcategory $\C_{pt}$ is Lagrangian. It
follows from \cite[Theorem 4.5]{drinfeld2007group} that $\C$ is equivalent to
the Drinfeld center of a fusion category of dimension $q^2$. Lemma \ref{pt-ddqq}
implies that $\C$ must be pointed, against the assumption. This contradiction
finishes the proof of the lemma. 
\end{proof}

\begin{remark} Observe that if $q$ is an odd prime number, then every modular
category $\C$ of Frobenius-Perron dimension $q^4$ is integral and therefore
pointed, by Lemma \ref{dim_q2}. On the other hand,  the Drinfeld center of an
Ising category provides an example of a non-pointed (and not integral) modular category of
Frobenius-Perron dimension $16$.   
\end{remark}

\begin{corollary}\label{pq4}
Let $\C$ be a modular category of Frobenius-Perron dimension $dq^4$, where $q$
is an odd prime number and $d$ is a square-free integer not divisible by $q$.
Then $\C$ is pointed. 
\end{corollary}

\begin{proof} Let $d = p_1 \dots p_s$, where $p_1, \dots, p_s$ are pairwise
distinct prime numbers, distinct from $q$.
By Theorem \ref{main}, $\C$ is integral and nilpotent. Therefore, in view of
\cite[Theorem 1.1]{drinfeld2007group}, $\C\cong
\C_{q^4}\boxtimes\C_{p_1}\boxtimes\cdots\boxtimes\C_{p_s}$, where $\C_{t}$ is a
modular category of dimension $t$ (see Remark \ref{rmk-mod-dec}).
In particular, the categories $\C_{p_1}, \dots, \C_{p_s}$ are pointed.
It follows from Lemma \ref{dim_q2} that $\C_{q^4}$ is also pointed, which
implies the statement.
\end{proof}

\section{Structure of a strictly weakly integral ASF modular
category}\label{ASF_fu_ca}
\setcounter{equation}{0}
Let $\D$ be a fusion category. If $\D$ is not pointed and $X\otimes Y$ is a
direct sum of invertible objects, for all non-invertible simple objects
$X,Y\in\D$, then $\D$ is called a generalized Tambara-Yamagami fusion category.
Generalized Tambara-Yamagami fusion categories were classified in
\cite{liptrap2010generalized}, up to equivalence of tensor categories. Recently,
they have been further studied in \cite{natale2013faithful}. In particular,
modular generalized Tambara-Yamagami fusion categories were classified in terms
of Ising modular categories and pointed modular categories.

\begin{proposition}\label{IsingTensorPointed}
Let $\D$ be a weakly integral  fusion category such that $\FPdim(\D_{ad})=2$.
Then
$\D$ is a generalized Tambara-Yamagami fusion category.
In particular, if $\D$ is modular, then $\D\cong\I\boxtimes\B$, where $\I$ is an
Ising fusion category and $\B$ is a pointed modular  category.
\end{proposition}
\begin{proof}
Since $\FPdim(\D_{ad})=2$, then $\D$ is not pointed.
Let $\D=\oplus_{g\in\U(\D)}\D_g$ be the universal grading of $\D$. Then
$\FPdim(\D_g)=2$ for all $g\in \U(\D)$. Since $\D$ is weakly integral and not
pointed, the FP dimension of every simple object is a square root of some
positive integer. Hence, every component $\D_g$ of the universal grading either
contains two non-isomorphic invertible objects, or it contains a unique
$\sqrt{2}$-dimensional simple object. Thus we get that the Frobenius-Perron
dimension of any simple object of $\D$ is $1$ or $\sqrt 2$. Hence $\C_{int}$ is
a pointed fusion category.

Considering the dimensional grading $\D=\oplus_{h\in E}\D_h$, we get that the
order of $E$ is $2$. We also get that $\D_0$ contains $1$-dimensional simple
objects, and $\D_1$ contains $\sqrt{2}$-dimensional simple objects. This implies
that $\D$ is a $\mathbb{Z}_2$-extension of a pointed fusion category. Let
$X,Y\in\D_1$ be 
simple objects.  Then $X\otimes Y\in\D_0$ is a direct sum of invertible objects.
This proves that $\D$ is a generalized Tambara-Yamagami fusion category.

If $\D$ is modular then \cite[Theorem 5.5]{natale2013faithful} (see also
\cite[Theorem 3.4]{2016Dongbraided}) shows that $\D\cong\I\boxtimes\B$ as
described.
\end{proof}

Let $\C$ be an ASF modular category of Frobenius-Perron dimension $dq^n$. In the
case when $\C$ is integral, the structure of $\C$ has been obtained in
\cite[Corollary 4.2]{2016DongIntegral}: $\C$ is either pointed, or equivalent to
a $G$-equivariantization of a nilpotent fusion category of nilpotency class $2$,
where $G$ is a $q$-group. Note in addition that if $\C$ is strictly weakly
integral, then $q = 2$, by Corollary \ref{2squarefree}.

\medbreak
For the rest of this section, $\C$ will denote a weakly integral ASF modular
category of Frobenius-Perron dimension $d2^n$.

\begin{lemma}\label{lem1}
The dimension of every simple object of $\C_{ad}$ is a power of $2$.
\end{lemma}
\begin{proof}
Let $X$ be a simple object of $\C_{ad}$. It is known that $\C_{ad}\subseteq
\C_{int}$ is an integral fusion subcategory of $\C$. Hence $\FPdim X$ is an
integer. On the other hand, $X$ is also a simple object of $\C$. It follows that
$\FPdim X^2$ divides $\FPdim \C=d2^n$ by \cite[Theorem 2.11]{etingof2011weakly}.
Hence, $\FPdim X$ is a power of $2$.
\end{proof}

The following lemma is a special case of \cite[Lemma 4.1]{bpr}:

\begin{lemma}\label{lem2} If $|\U(\C)|$ is divisible by $2^n$ then $\C$ is
pointed.
\end{lemma}

\begin{lemma}\label{lem3}
Assume that $\C$ is not pointed. Then $(\C_{ad})_{pt}$ is not trivial.
\end{lemma}
\begin{proof}
By Lemma \ref{lem2}, $|\U(\C)|$ is not divisible by $2^n$. Hence  $\FPdim
\C_{ad} $ is divisible by $2$. By Lemma \ref{lem1}, the dimension of every
simple object of $\C_{ad}$ is a power of $2$. Let $1,2,2^2,\cdots,2^t$ be all
possible dimensions of simple objects of $\C_{ad}$, and let
$a_0,a_1,a_2,\cdots,a_t$ be the number of non-isomorphic simple objects of these
dimensions. Then we have $a_0+a_12^2+a_22^4+\cdots+a_t2^{2t}=\FPdim \C_{ad} $.
Since the right hand side is divisible by $2$, we get that $a_0$ is also
divisible by $2$. This shows that $(\C_{ad})_{pt}$ is not trivial.
\end{proof}

\begin{proposition}\label{ExistNntrivTannSub}
Assume that $\C$ is not pointed and $\FPdim \C_{ad} >2$. Then $\C$ contains a
non-trivial Tannakian subcategory.
\end{proposition}
\begin{proof}
By Lemma \ref{lem3} and \cite[Lemma 2.4]{2016GonNatale}, $(\C_{ad})_{pt}$ is a
non-trivial symmetric fusion category. If $\FPdim (\C_{ad})_{pt}>2$ then
$(\C_{ad})_{pt}$ contains a non-trivial Tannakian subcategory with
Frobenius-Perron dimension $\FPdim (\C_{ad})_{pt}/2$(see the exposition in
Section \ref{sec2.3}). Hence we may assume in the rest that $\FPdim
(\C_{ad})_{pt}=2$.

\medbreak
Since we have assumed that $\FPdim \C_{ad} >2$, $\C_{ad}$ contains
non-invertible simple objects. Let $\delta$ be the unique non-trivial invertible
simple object in $(\C_{ad})_{pt}$. Then $\delta$ appears in the decomposition of
$X\otimes X^*$ for every non-invertible $X\in \Irr(\C_{ad})$. To see this, we
just notice that the Frobenius-Perron dimension of every non-invertible simple
object of $\C_{ad}$ is a power of $2$, by Lemma \ref{lem1}. Therefore, we have

\begin{equation}\label{eqdelta}
\begin{split}
\delta\otimes X\cong X, \quad X\in \Irr(\C_{ad})/\{1,\delta\}.
\end{split}
\end{equation}

Since $\delta \in \C_{pt} = \C_{ad}'$, it follows from  \cite[Lemma
5.4]{muger2000galois} that $(\C_{ad})_{pt}$ cannot be the category $\svect$ of
super vector spaces. Hence, $(\C_{ad})_{pt}$ must be a Tannakian subcategory.
\end{proof}

\begin{theorem}\label{symm}
Let $\C$ be a strictly weakly integral ASF modular category. Then $\C$ fits into
one of the following classes:

(1)\, $\C$ is equivalent to a Deligne tensor product $\I\boxtimes\B$, where $\I$
is an Ising fusion category and $\B$ is a pointed modular  category.

(2)\, $\C$ is equivalent to a $G$-equivariantization of a braided $G$-crossed 
fusion category
$\oplus_{g\in G}\D_g$, where $G$ is a $2$-group and $\D_e$ is a pointed modular 
category.

(3)\, $\C$ is equivalent to a $G$-equivariantization of
a braided $G$-crossed  fusion category
$\oplus_{g\in G}\D_g$, where $G$ is a $2$-group and $\D_e\cong\I\boxtimes\B$,
where $\I$ is an Ising fusion category and $\B$ is a pointed modular  category.
\end{theorem}
\begin{proof}
Suppose that $\FPdim \C_{ad} =2$. Then (1) holds true by Proposition
\ref{IsingTensorPointed}.  In the rest of our proof, we consider the case where
$\FPdim \C_{ad} >2$.

By Proposition \ref{ExistNntrivTannSub}, $\C$ has a non-trivial Tannakian
subcategory. Let $\D\cong \Rep(G)$ be a proper maximal Tannakian subcategory of
$\C$, and let $\D'$ be its  M\"{u}ger centralizer in $\C$. Then $\D$ is the
M\"{u}ger center of $\D'$. Let $\E=(\D')_G$ be the de-equivariantization of
$\D'$ by $\Rep(G)$. It is non-degenerate by \cite[Remark
2.3]{etingof2011weakly}. Since the dimension of $\D$ is a power of $2$  and
$\FPdim \E=\FPdim \C/4$,  we know that $\E$ is also an ASF modular category.

Let $\C_G$ be the de-equivariantization of $\C$ by $\Rep(G)$. Then
$\C_G=\oplus_{g\in G}(\C_G)_g$ has a faithful $G$-grading and the neutral
component $(\C_G)_e$ is non-degenerate. By \cite[Proposition
4.56]{drinfeld2010braided}, $(\C_G)_e\cong (\D')_G=\E$.

The fusion category $\E$ is the core of $\C$ in the sense of \cite[section
5.4]{drinfeld2010braided}. It is a weakly anisotropic braided fusion category by
\cite[Corollary 5.19]{drinfeld2010braided}. By \cite[Corollary
5.29]{drinfeld2010braided}, $\E_{pt}\cap(\E_{pt})'$ is either equivalent to
$\vect$, or to the category $\svect$ of super vector spaces.

If $\E_{pt}\cap(\E_{pt})'\cong\vect$ then $\E$ is pointed. In fact, if $\E$ is
not pointed then the proof of Proposition \ref{ExistNntrivTannSub} shows that
$\E_{pt}\cap(\E_{pt})'=\E_{pt}\cap \E_{ad}=(\E_{ad})_{pt}$ contains a
non-trivial $G$-stable  Tannakian subcategory. In this case, $\C\cong (\C_G)^G$
fits into the second class.

If $\E_{pt}\cap(\E_{pt})'\cong\svect$ then $(\E_{ad})_{pt}=\E_{pt}\cap
\E_{ad}=\E_{pt}\cap(\E_{pt})'$ has Frobenius-Perron dimension $2$. By the proof
of Proposition \ref{ExistNntrivTannSub}, we get that $\FPdim \E_{ad}=2$. Hence
Proposition \ref{IsingTensorPointed} shows that $\E$ is equivalent to a Deligne
tensor product $\I\boxtimes\B$, where $\I$ is an Ising fusion category and $\B$
is a modular pointed category.  In this case, $\C\cong (\C_G)^G$ fits into the
third class.
\end{proof}

\begin{remark} Let $n \geq 2$. Examples of strictly weakly integral ASF modular
categories of Frobenius-Perron dimension $d2^{2n}$, $d \geq 1$, which are not in
the class described in Theorem \ref{symm}~(1), are provided by the tensor
products
\begin{equation*}\I_1 \boxtimes \dots \boxtimes \I_n \boxtimes \B,
\end{equation*} where $\I_1, \dots, \I_n$ are Ising modular categories and $\B$
is a pointed modular category of dimension $d$.

Consider the case where $n = 2$ and $\B$ is trivial, that is, we have a tensor
product $\C = \I_1 \boxtimes \I_2$ of two Ising modular categories. It follows
from \cite[Lemma B.24]{drinfeld2010braided}  that $\C$ is equivalent to a
$\Zz_2$-equivariantization of a braided $\Zz_2$-crossed  fusion category
whose trivial homogeneous component is a pointed modular  category corresponding
to a certain metric group of order 4; hence $\C$ belongs to the class described
in Theorem \ref{symm}~(2). 
\end{remark}

\section{Acknowledgements}
The first author was partially supported by the Fundamental Research Funds for
the Central Universities (KYZ201564), the Natural Science Foundation of China
(11201231) and the Qing Lan Project. The second author was partially supported
by  CONICET and SeCYT--UNC

\end{document}